\documentclass{amsart}

\usepackage{amsfonts}
\usepackage{amssymb}
\usepackage{amscd}

\newtheorem{thm}{Theorem}[section]
\newtheorem{lem}[thm]{Lemma}
\newtheorem{prop}[thm]{Proposition}

\newtheorem{cor}[thm]{Corollary}
\newtheorem{conj}[thm]{Question}
\theoremstyle{definition}
\newtheorem{dfn}[thm]{Definition}
\newtheorem{rmk}[thm]{Remark}
\newtheorem{ex}[thm]{Example}

\begin{document}

\title[Bounded modular functionals and operators with unusual kernels]{On bounded module functionals and operators on Hilbert C*-modules without biorthogonally complemented kernels}
\author{Michael Frank}
\address{Hochschule f\"ur Technik, Wirtschaft und Kultur (HTWK) Leipzig, Fakult\"at Informatik und Medien, PF 301166, D-04251 Leipzig, Germany.}
\email{michael.frank@htwk-leipzig.de}

\author{Cristian Ivanescu}
\address{MacEwan University, Department of Mathematics and Statistics, 10700 104 Avenue NW, Edmonton, AB T5J 4S2, Canada}
\email{ivanescuc@macewan.ca}



\date{August, 2026}

\begin{abstract} 
We investigate the structural reasons behind the existence and non-existence of bounded module functionals and bounded module operators on Hilbert C*-modules whose kernels are not orthogonally complemented. The key motivating example was introduced by J.~Kaad and M.~Skeide in their 2023 paper. Building on their ideas and related developments, we obtain a structural characterization of this class of examples in a more abstract theoretical framework. We also study non-adjointable bounded module operators associated with such module functionals, together with the properties of the corresponding operator algebras and operator spaces. Our results provide a deeper understanding of phenomena beyond the adjointable setting, place several existing partial results into a unified framework, and show how large families of further examples can be constructed systematically.
\end{abstract}

\maketitle

\section{Introduction}

Since 2015, there have been considerable discussions about the possible existence of non-regular bounded modular functionals on certain Hilbert C*-modules: for two Hilbert C*-modules ${\mathcal M} \subset \mathcal N$ over some fixed C*-algebra $A$ with ${\mathcal M}^\bot = \{ 0 \}$, does there exist a non-zero bounded $A$-linear map $r: {\mathcal N} \to A$ such that $\mathcal M$ belongs to the kernel of $r$, or not. The problem has been raised by O.~M.~Shalit and M.~Skeide in \cite[Quest. 10.6, 10.7]{Shalit_Skeide} in 2003/2020/2023, pointing out important open problems in Hilbert product systems embedding research in semi-group contexts. Already in 2015 the problem was mentioned in \cite[p.~1546, Footnote 3]{BhSk_2015}. In \cite[Lemma 2.4]{Frank_2002}, the first author stated the identity of the biorthogonal complement of the kernels of general bounded Hilbert module operators, however the proof presented there is insufficient. D.~Baki\'c and B.~Gulja{\v{s}} gave a negative answer for ideal submodules in \cite{BG_2002} in 2002. A first affirmative example was constructed by J.~Kaad and M.~Skeide \cite{KaSk} in 2021. Further affirmative examples were announced by M.~V.~Manuilov in \cite{M_2022} in 2024. M.~Skeide found a negative answer for closed ternary ideals in \cite{Sk_2022} in 2022. Later, one of the authors of the present paper was able to show the non-existence of such bounded modular functionals in the described situations for von Neumann and monotone complete C*-algebras of coefficients in \cite{Frank_2024}, in contrast to the examples in \cite{M_2022}. This saved \cite[Lemma 2.4]{Frank_2002} for von Neumann and monotone complete C*-algebras of coefficients, and hence, the content of that paper \cite{Frank_2002}. However, the structural background of the appearing counterexamples is not well understood at present. Moreover, for Hilbert C*-modules and their multiplier modules a negative answer was found in \cite{Frank_2026}, where, the more specifically, not each bounded module map of the smaller Hilbert C*-module to $A$ can be continued to a bounded module map of its multiplier module to the multiplier algebra $M(A)$ of $A$, in cases, cf.~\cite{Frank_2026}. 

From a categorical point of view, the equivalent question is whether adequate bounded modular functionals on Hilbert C*-modules always separate Hilbert C*-submodules in Hilbert C*-modules, thereby making them distinguishable. From another point of view, one asks for non-trivial extensions of the zero modular functional from Hilbert C*-submodules to hosting Hilbert C*-modules. Also, there is an equivalent formulation of that problem in terms of the existence of non-regular bounded modular operators on the hosting Hilbert C*-module, cf.~\cite[Prop.~4.2, Prop.~4.3]{Frank_2024} and \cite{KaSk}. Note, that both these questions below have the answer `never' for adjointable bounded C*-linear operators by Lemma \ref{case_selfadjoint}, due to their biorthogonally complemented kernels.

\begin{conj} \label{conj_1}
Let $A$ be a C*-algebra and let $\mathcal M$ be a Hilbert $A$-submodule of a Hilbert $A$-module $\mathcal N$ such that the orthogonal complement of $\mathcal M$ relative to $\mathcal N$ is trivial,  i.e. equals $\{0\}$. Under which conditions does there exist any non-trivial bounded $A$-linear map from $\mathcal N$ to $A$ such that is equal to the zero map on $\mathcal M$? 
\end{conj}
 
\begin{conj} \label{conj_2}
Let $A$ be a C*-algebra and let $\mathcal M$ be a Hilbert $A$-submodule of a Hilbert $A$-module $\mathcal N$ such that the orthogonal complement of $\mathcal M$ relative to $\mathcal N$ is trivial,  i.e. equals $\{0\}$. Under which conditions does there exist any non-trivial bounded $A$-linear map from $\mathcal N$ to $\mathcal N$ such that is equal to the zero map on $\mathcal M$? 
\end{conj}

In the case of essential norm-closed ideals of commutative C*-algebras a central role is played by Urysohn's lemma (\cite[Thm.~2.12]{R}) and its extension to a dichotomy. Consider a unital commutative C*-algebra $A=C(K)$, where $K$ is a compact Hausdorff space, and a norm-closed ideal $I \subset A$. The ideal $I$ is characterizable by a subset $L \subset K$ of points $t$ of $K$ such that at least one element $i \in I$ evaluated at $t \in L$ is unequal to zero. The Urysohn's lemma states that for a normal topological space $X$ and for two disjoint closed sets $L_1$, $L_2$ of it, there exists a continuous function $f$ on $X$ such that $f|_{L_1} = 0$, $f|_{L_2} =1$ and $0 \leq f(t) \leq 1$ for $t \in X$. So, Question \ref{conj_1} on norm-closed ideals $I$ such that $I^\bot = \{ 0 \}$ translates into the assumption that $K \setminus L$ does not contain any closed set disjoint from the closure of $L$ in $K$, even not a single point. The closure of $L$ should be equal to $K$. In algebraic terms, we can rephrase this situation requiring that the ideal $I$ has to be an essential ideal of $A$, i.e.~$ai=0$ for any $i \in I$ and a certain $a \in A$ forces $a=0$. (In fact, $a$ may be even selected from the multiplier algebra $M(I)$.) We can consider Question \ref{conj_1} as an extension of the assertion of Urysohn's lemma: For a given closed subset $L_1$ of a compact Hausdorff space $K$ either there exists another non-trivial disjoint closed subset $L_2 \subset K$ or $L_1=K$. In the case of essential ideals $I \subset A$, the closure of the set $L$, which corresponds to $I$, is equal to $K$.

Considering commutative non-unital C*-algebras $A=C_0(K)$, where $K$ is merely a locally compact Hausdorff space, and essential ideals $I$ of them we can switch to the unitization $A_1$ of $A$, i.e.~to the one-point compactification of $K$. Again, the ideal $I$ is characterizable by a subset $L \subset K$ of points $t$ of $K$ such that at least one element $i \in I$ evaluated at $t \in L$ is unequal to zero. Then $I$ is still an essential ideal in $A_1$. The switch to the unitization $A_1$ of $A$ would be useful in some cases because of the existence of non-normal locally compact Hausdorff spaces like the Tychonoff plank, cf.~\cite{StSe}. A variant of Urysohn's lemma for locally compact Hausdorff spaces, even non-normal, can be found in the book by W.~Rudin, cf.~\cite[Thm.~2.12]{R}. There we can rely on the complete regularity property: a topological Hausdorff space is completely regular if for every closed set $F$ and for every single point $x$ outside $F$ there exists a continuous real-valued function $f$ on the space with values in $[0,1]$ such that $f$ equals zero on $F$ and equals one on $x$. Since locally compact Hausdorff spaces are completely regular (\cite[17.14d, p.~460]{Sch}, \cite{Wiki}), the existence of just one point outside the closure of $L \subset K$ would give a continuous function on $K$ such that it is equal to zero on the closure of $L$ and equal to one on that point, contradicting the essentiality property of $I$. So the closure of $L$ relative to $K$ has to be equal to $K$, again.

In the non-commutative case, the consideration of two-sided essential ideals of C*-algebras is similar, however the canonically generalized objects would be {one-sided} norm-closed ideals with the module action from one side (the ideal side), but with the modular operator action from the (not specified) other side. So, the situation has to be investigated for general Hilbert C*-modules, in the commutative case as well.


\section{Preliminaries}

Let $A$ be a C*-algebra. We consider Hilbert C*-modules $\mathcal K$ as left modules over $A$ with a module action compatible with both the complex linear space structures. Hilbert C*-modules possess an $A$-valued inner product $\langle \cdot , \cdot \rangle$ $A$-linear in the first entry, positive in the sense that $\langle x, x\rangle \geq 0$ for all $x$, and $\langle x, x \rangle=0$ implies $x=0$. In addition, the $A$-valued inner product is
 compatible with the involution on $A$, interchanging the two entries. We always assume that $\mathcal K$ is complete with respect to the norm defined by $\| \cdot \| = \| \langle \cdot , \cdot \rangle \|^{1/2}$. Usually, one has to consider a Hilbert C*-module as a pair $\{ {\mathcal K}, \langle \cdot , \cdot \rangle \}$ because there are examples with two different $A$-valued inner products inducing equivalent norms, but which are not unitarily isomorphic as $A$-valued inner products, cf.~\cite{F_1999_1}. However, for our purposes, just the existence of an $A$-valued inner product is sufficient. 

Bounded C*-linear operators are thought of as acting from the right in the sense that they commute with the left module action. Instructive examples are left {norm-complete} essential ideals $I$ of C*-algebras $A$ for which any bounded module operator is realized as a right multiplication by a right multiplier of $A$. The operator norm is defined as usual for bounded linear operators. A bounded (C*-linear) operator $T$ from a Hilbert C*-module ${\mathcal X}$ to a Hilbert C*-module ${\mathcal Y}$ is called adjointable if the equality $\langle T(x),y \rangle = \langle x,T^*(y) \rangle$ is valid for some bounded C*-linear operator $T^*: {\mathcal Y} \to {\mathcal X}$ (the adjoint operator of $T$) and any $x \in \mathcal X$, $y \in \mathcal Y$. Note, that boundedness and adjointability already imply C*-linearity. Cf.~\cite[Section 3]{Lance} for complementary discussions.

\begin{lem} \label{case_selfadjoint}
Let $A$ be a C*-algebra and let ${\mathcal X}$ and ${\mathcal Y}$ be Hilbert C*-modules. For every adjointable bounded $A$-linear operator
$T:{\mathcal X}\to{\mathcal Y}$, exactly one of the following three situations occurs:
\begin{itemize}
\item[(i)]
The kernel ${\rm Ker}(T)$ of $T$ is biorthogonally complemented, its orthogonal complement is equal to the biorthogonal closure ${\rm Im}(T^*)^{\perp\perp}$ of the range of $T^*$, and the orthogonal direct sum
\[
{\rm Ker}(T)\oplus {\rm Im}(T^*)^{\perp\perp} \subset \mathcal X
\]
is a proper Hilbert $A$-submodule of ${\mathcal X}$.

\item[(ii)]
The kernel ${\rm Ker}(T)$ of $T$ is biorthogonally complemented, its orthogonal complement is equal to the norm closure $\overline{{\rm Im}(T^*)}$, the orthogonal direct sum
\[
{\rm Ker}(T)\oplus \overline{{\rm Im}(T^*)} = \mathcal X
\]
equals ${\mathcal X}$, but the range ${\rm Im}(T^*)$ is not norm closed.

\item[(iii)]
The kernel ${\rm Ker}(T)$ of $T$ is biorthogonally complemented, its orthogonal complement is equal to the range ${\rm Im}(T^*)$ of $T^*$, and
\[
{\rm Ker}(T)\oplus {\rm Im}(T^*) ={\mathcal X} \, .
\]
\end{itemize}
Only in cases (ii) and (iii) $T$ admits a polar decomposition.
\end{lem}

\begin{proof}
Assertions (ii) and (iii) follow from
\cite[Thms.~15.3.7 and 15.3.8]{NEWO}, where the existence of the polar decomposition is also established.
To prove (i), let 
\linebreak[4]
$x \in {\rm Ker}(T)$. Then, for every $y \in \mathcal Y$,
\[
0 = \langle y,T(x)\rangle =\langle T^*(y),x\rangle \, ,
\]
which shows that ${\rm Ker}(T)= {\rm Im}(T^*)^\perp$. Consequently, ${\rm Ker}(T)$ is biorthogonally complemented and
\[
{\rm Ker}(T)^\perp = {\rm Im}(T^*)^{\perp\perp} \, .
\]

To see that situation (i) may indeed occur, let $A=C([0,1])$ and identify
\linebreak[4]
$A=\mathcal X=\mathcal Y$ with standard Hilbert $A$-module structures. Consider the self-adjoint multiplication operator $T_g(f)(t)=g(t)f(t)$, where
\[
g(t)=
\begin{cases}
        1-2t, & 0\le t\le \frac12 \, ,\\
         0, & \frac12\le t\le 1 \, .
\end{cases}
\]
Then
\[
{\rm Ker}(T_g) = \{f\in C([0,1]):f|_{[0,1/2]}=0\} \, ,
\]
while
\[
{\rm Im}(T_g) = \{f\in C([0,1]):f|_{[1/2,1]}=0\} \, .
\]
Both ideals are biorthogonally complemented. However,
\[
{\rm Ker}(T_g)\oplus {\rm Im}(T_g) = \{f\in C([0,1]):f(1/2)=0\} \subset A
\]
which is a proper ideal of $A$. Therefore, case (i) can occur.
\end{proof}

\section{On the regularity of bounded modular functionals and operators}

The aim of this section is to give more insight into the background of these two Questions \ref{conj_1}, \ref{conj_2}. We build on the results known for von Neumann algebras $A$ and for analogous pairs of Hilbert $A$-modules ${\mathcal M} \subseteq \mathcal N$ worked out in \cite[Thm.~3.8, Prop.~4.2, Prop.~4.3]{Frank_2024}. For such settings, the questions are known to have a negative answer. We aim to use the universal $*$-representation of C*-algebras $A$ in their respective bidual von Neumann algebra $A^{**}$, together with a tensor product extension construction of Hilbert $A$-modules to canonically related Hilbert $A^{**}$-modules invented by W.~L.~Paschke in \cite{Paschke}. This will give a basis to apply the known result for the von Neumann case to the general setting.

There is one circumstance to pay attention to. We show the phenomenon by example.

\begin{ex} \label{ex-max-ideal}
Let $A$ be a C*-algebra, and $I$ be a maximal norm-closed one-sided ideal in $A$. Maximal one-sided (say, left) norm-closed ideals of C*-algebras $A$ are in one-to-one correspondence to minimal projections $p$ in their bidual von Neumann algebra $A^{**}$, identifying $I \subset A$ by $I=A(1-p) \subset A \subseteq A^{**}$. Obviously, these projections belong to the type I part of $A^{**}$.  Now, consider $I$ as a left Hilbert $A$-submodule of the Hilbert $A$-module $A$, and its extension $A^{**}(1-p)$ as a left Hilbert $A^{**}$-submodule of the Hilbert $A^{**}$-module $A^{**}$. The orthogonal complement of $I$ with respect to $A$ is $\{ 0 \}$. However, the orthogonal complement of $I \subset A^{**}(1-p)$ with respect to $A^{**}$ is $A^{**}p$, i.e. it is non-trivial. So, the canonical extension from a pair of Hilbert C*-modules ${\mathcal M} \subseteq \mathcal N$ with ${\mathcal M}^\bot = \{ 0 \}$ to a canonical pair of Hilbert $A^{**}$-modules following W.~L.~Paschke's tensor product construction in \cite{Paschke} might not preserve the trivial orthogonal complement property of the smaller module with respect to the larger module. Note, that the right action of $p$ on $A$ or on $A^{**}$ can be considered as an action of a bounded module operator. However, the Questions \ref{conj_1}, \ref{conj_2} above have a negative answer by \cite[Prop.~4.2, Prop.~4.3, Thm.~6.1]{Frank_2024}.

The trivial orthogonal complement property might also be preserved in cases. Let $A$ be a von Neumann algebra and consider ${\mathcal M} = \ell_2(A)$ and ${\mathcal N}=\ell_2(A)'$, which are both Hilbert $A$-modules by \cite[Thm.~3.2]{Paschke}, the latter self-dual. Obviously, $\ell_2(A)^\bot = \{ 0 \}$ with respect to $\ell_2(A)'$. Then still $\ell_2(A^{**})^\bot = \{ 0 \}$ with respect to $\ell_2(A^{**})'$. 
\end{ex}

For Hilbert W*-modules, the Questions \ref{conj_1} and \ref{conj_2} have a negative answer:

\begin{lem} \label{thm_von_N} {\rm (Cf.~\cite[Thm.~3.8, Prop.~4.2, Prop.~4.3]{Frank_2024}.)} \newline
Let $A$ be a W*-algebra. Let $\mathcal M$ be a Hilbert $A$-submodule of a Hilbert $A$-module $\mathcal N$ such that the orthogonal complement of $\mathcal M$ relative to $\mathcal N$ is trivial,  i.e. equals $\{0\}$. Then there does not exist any non-trivial bounded $A$-linear map from $\mathcal N$ to $A$ such that it equals the zero map on $\mathcal M$. Similarly, there does not exist any non-trivial bounded $A$-linear map from $\mathcal N$ to $\mathcal N$ such that it equals the zero map on $\mathcal M$. 
\end{lem}

We want to follow the work by W.~L.~Paschke \cite{Paschke} to transfer properties from Hilbert C*-modules to Hilbert W*-modules over the bidual von Neumann algebra.
Let $A$ be any C*-algebra with bidual von Neumann algebra $A^{**}$. We identify $A$ with its canonical embedding via the universal $*$-representation of $A$ in $A^{**}$ without special notations. So they share the identity in that picture.  

For any Hilbert $A$-module $\mathcal K$, one can form the tensor product ${\mathcal K} \odot_A A^{**}$, identifying elementary tensors $ax \otimes b$ and $x \otimes ba$ for any $a \in A$, $b \in A^{**}$, $x \in \mathcal K$, cf. \cite[pp.~459-460, 463]{Paschke}. On elementary tensors, the $A$-valued inner product on $\mathcal K$ can be extended by setting 
\[
      \langle x \otimes a, y \otimes b \rangle = \langle ax \otimes 1_{A^{**}}, by \otimes 1_{A^{**}} \rangle = a \langle x \otimes 1_A, y \otimes 1_A \rangle b^* = a \langle x,y \rangle b^* 
\]    
for any $x,y \in \mathcal K$, any $a,b \in A^{**}$.  Denote the norm-completion of ${\mathcal K} \odot_A A^{**}$ by ${\mathcal K}^{\#}$, Lifting the $A$-valued inner product from $\mathcal K$ to an $A^{**}$-valued inner product on ${\mathcal K}^{\#}$ we get the structure of a Hilbert $A^{**}$-module.
Consider a bounded module map $T: {\mathcal K} \to \mathcal K$ and its extension to ${\mathcal K}^{\#}$ by the formula $T(x \otimes b) = T(x) \otimes b$ for any $x \in \mathcal K$, $b \in A^{**}$. We get a bounded $A^{**}$-linear map on ${\mathcal K}^{\#}$ of the same operator norm extending $T$, also denoted by $T$.

Following W.~L.~Paschke's arguments in \cite{Paschke} we can extend the $A^{**}$-valued inner product from ${\mathcal K}^{\#}$ to its $A^{**}$-dual Banach $A^{**}$-module $({\mathcal K}^{\#})'$, turning the latter into a self-dual Hilbert $A^{**}$-module. Also, there is a unique bounded $A^{**}$-linear extension $T: ({\mathcal K}^{\#})' \to ({\mathcal K}^{\#})'$ of any bounded $A^{**}$-linear operator $T: {\mathcal K}^{\#} \to {\mathcal K}^{\#}$ preserving its operator norm, cf.~\cite[Thm.~3.2, Prop.~3.6, Cor.~4.3]{Paschke}. Reconsidering Example \ref{ex-max-ideal}, for ${\mathcal K}=I$ we get $({\mathcal K}^{\#})' = A^{**}(1-p)$, and not $A^{**}$.

Motivated by Urysohn's lemma, we first consider a situation with certain self-adjoint operators $S$.

\begin{lem} \label{lemma_sa}
Let $A$ be a C*-algebra with bidual von Neumann algebra $A^{**}$. Consider two Hilbert $A$-modules ${\mathcal M} \subseteq {\mathcal N}$ with ${\mathcal M}^\bot=\{0\}$ relative to $\mathcal N$. We are especially interested in the case of non-coincidence of these two modules. Let $S: {\mathcal N} \to {\mathcal N}$ be a self-adjoint bounded $A$-linear operator with $S\big|_{\mathcal M}=0$ and $S\neq 0$.

Form the Paschke extensions ${\mathcal M}^{\#}$, ${\mathcal N}^{\#}$ and their selfdual completions $({\mathcal M}^{\#})' \subseteq ({\mathcal N}^{\#})'$ as in \cite{Paschke}. Let $P: ({\mathcal N}^{\#})' \to ({\mathcal M}^{\#})'$ be the orthogonal projection coming from the selfdual Hilbert $A^{**}$-module structure.

Then, considering the canonical bounded $A^{**}$-linear extension $S: ({\mathcal N}^{\#})'\to ({\mathcal N}^{\#})'$, we obtain a contradiction to the von Neumann case of Lemma \ref{thm_von_N}  as proved in \cite[Thm.~3.8]{Frank_2024}. In particular, such a self-adjoint bounded $A$-linear operator $S: {\mathcal N} \to {\mathcal N}$ with $S \not= 0$ cannot exist, $S=0$.
\end{lem}

\begin{proof}
By Paschke's construction \cite{Paschke}, ${\mathcal M}^{\#}$ and ${\mathcal N}^{\#}$ are Hilbert $A^{**}$-modules, and the given operator $S: {\mathcal N} \to {\mathcal N}$ admits a unique bounded $A^{**}$-linear extension (still denoted $S$) of equal norm to ${\mathcal N}^{\#}$ and further to the selfdual Hilbert $A^{**}$-module $({\mathcal N}^{\#})'$; see \cite[Thm.~3.2, Prop.~3.6, Cor.~4.3]{Paschke} and \cite[Thm.~3.3]{Frank_2024}. The inclusion ${\mathcal M}\subseteq {\mathcal N}$ yields $({\mathcal M}^{\#})' \subseteq ({\mathcal N}^{\#})'$, and since $({\mathcal M}^{\#})'$ and $({\mathcal N}^{\#})'$ are selfdual by construction, $({\mathcal M}^{\#})'$ is an orthogonal direct summand of $({\mathcal N}^{\#})'$. Let $P: ({\mathcal N}^{\#})' \to ({\mathcal M}^{\#})'$ be the corresponding orthogonal projection, which could be unequal to the identity map.

By construction, $S$ vanishes on $\mathcal M$, and hence, on ${\mathcal M}^{\#}$ and on $({\mathcal M}^{\#})'$ because, as a bounded $A^{**}$-linear operator on the self-dual Hilbert $A^{**}$-module $({\mathcal N}^{\#})'$, it has an orthogonally complemented kernel. On the other hand, the operator $S$ is non-zero on $\mathcal N$ by assumption, so also on $({\mathcal N}^{\#})'$, and $S = S \circ P + S \circ (\mathrm{id}-P)$. By \cite[Thm.~3.8, Prop.~4.3]{Frank_2024}, for pairs of Hilbert $A^{**}$-modules with trivial orthogonal complement in the von Neumann algebra setting, there is no non-zero bounded $A^{**}$-linear operator that vanishes on the smaller module. More precisely, applied to the pair $({\mathcal M}^{\#})' = P({\mathcal N}^{\#})' \subseteq ({\mathcal N}^{\#})'$ and to the restriction of $S$ to $P({\mathcal N}^{\#})')$, the results of \cite{Frank_2024} imply that $S|_{P({\mathcal N}^{\#})'}=0$. This gives also $S\circ P=0$ and $S= S \circ (\mathrm{id}-P)$. By the self-adjointness of $S$, we have also $({\mathrm id}-P) \circ S = S$ and $P \circ S = 0$. 

Furthermore, the orthogonal complement $({\mathcal M}^{\#})'^\bot$ inside $({\mathcal N}^{\#})'$ is also a selfdual Hilbert $A^{**}$-module, 
and the restriction $S\big|_{({\mathcal M}^{\#})'^\bot}$ is a bounded $A^{**}$-linear operator on it. 

Since $P \circ S = S\circ P=0$, we have $S = ({\mathrm id}-P) \circ S\circ (\mathrm{id}-P)$ on $({\mathcal N}^{\#})'$. In particular, since $S\neq 0$ is assumed on ${\mathcal N}$, there exists $n\in {\mathcal N}\subset ({\mathcal N}^{\#})'$ with $S(n)\neq 0$. Then $S(n)=({\mathrm id}-P) \circ S \circ (\mathrm{id}-P)(n)\in ({\mathcal M}^{\#})'^\bot$,  and $S(n)\in {\mathcal N}$ by assumption. Thus $S(n)\in {\mathcal N}\cap ({\mathcal M}^{\#})'^\bot$.

However, ${\mathcal N}\cap ({\mathcal M}^{\#})'^\bot\subset {\mathcal M}^\bot$ (with respect to ${\mathcal N}$),  because the inner product on ${\mathcal N}$ is the restriction of that on $({\mathcal N}^{\#})'$, and respectively for ${\mathcal M}\subset ({\mathcal M}^{\#})'$.  Since $S(n)\neq 0$ and $S(n)\in {\mathcal M}^\bot$, this contradicts ${\mathcal M}^\bot=\{0\}$ in ${\mathcal N}$.  Thus $({\mathrm id}-P) \circ S \circ (\mathrm{id}-P)=0$. 

Combining $S\circ P=0$ and $S\circ(\mathrm{id}-P)=0$, we obtain $S=0$ on $({\mathcal N}^{\#})'$, hence on ${\mathcal N}$, contradicting $S\neq 0$.  Therefore, no such self-adjoint $S$ can exist.
\end{proof}

\begin{rmk}
Bounded adjointable operators on Hilbert C*-modules admit polar decomposition if and only if the norm-closure of their range as well as their kernels are direct orthogonal summands, \cite[Thms.~15.3.7, 15.3.8]{NEWO}. Their kernels are always biorthogonally complemented. It is easy to find examples of adjointable operators on Hilbert C*-modules, the range of which is neither norm-complete nor the {norm-completion} of it is biorthogonally complemented, because there might not exists any orthogonal projection operators on the Hilbert C*-module besides the zero and the identity operator, like for $A={\mathcal X}=C([0,1])$.  
However, non-adjointable module operators may admit kernels and the norm-closure of their ranges, which are not direct orthogonal summands anymore.  Consider a (unital) C*-algebra $A$ and a one- or two-sided norm-closed ideal $I$ in it. Form the Hilbert $A$-module $A \oplus I$ as a submodule of the canonical Hilbert $A$-module $A^2$. Consider the skew projection $P: A \oplus I \to A \oplus I$ defined by the formula $P((a,i))=(i,i)$ for $a \in A$, $i \in I$. Then the range $P(A \oplus I)=\{ (i,i) : i \in I \}$ is a norm-closed Hilbert $A$-submodule and even coincides with its biorthogonal complement, but it is only a topological direct summand with topological complement $(\mathrm{id}-P)(A \oplus I)= \{ (a-i,0) : a \in A, i \in I \}$. Its orthogonal complement is the submodule $P(A \oplus I)^\bot=\{ (-j,j) : j \in I \}$, which is also only a topological direct summand, but it is biorthogonally complemented, too. Note, that there exists a non-negative operator $S$ on $A \oplus I$ such that $\langle x,S(y) \rangle = \langle P(x),P(y) \rangle$ for any $x,y \in A \oplus I$, even $P$ is a non-adjointable bounded module operator, cf.~\cite{Sk_2025}. The operator $S$ equals two times the direct orthogonal projection onto the second direct orthogonal summand of $A \oplus I$. The eigenvalues of $S$ are $\{ 0, 2 \cdot \mathrm{id} \} \in A$ with obvious eigen-submodules $A \oplus \{ 0 \}$ and $\{ 0 \} \oplus I$. 
\end{rmk}

The situation becomes more complicated if $S$ is not assumed to be self-adjoint. The example by J.~Kaad and M.~Skeide in \cite{KaSk} is of that kind. 

\begin{lem}\label{lem:3.5}
Assume the situations and constructions described at Lemma \ref{lemma_sa}. We consider operators $S$ not necessarily self-adjoint. Let $Q: ({\mathcal N}^{\#})' \to {\rm Ker}(S)$ be the corresponding orthogonal projection, $Q \geq P$. 
\begin{itemize}
\item[(i)]  If $P=\mathrm{id}$ then $S=0$.
\item[(ii)] If $P \circ S = 0$ or $Q \circ S = 0$ then $S=0$. This happens in particular if $S$ and $P$ or $Q$ commute. 
\end{itemize}
Consequently, the existence of a non-trivial bounded $A$-linear operator $S$ on $\mathcal N$ with $S \big|_{\mathcal M} = 0$ forces $P \circ S \not = 0$, $Q \circ S \neq 0$,  $({\mathrm id}-P) \circ S \neq 0$ and $({\mathrm id}-Q) \circ S \neq 0$, together with $S \circ P = S \circ Q =0$, i.e.~$S = S \circ ({\mathrm id}-Q) = S \circ ({\mathrm id}-P)$.
\end{lem}

\begin{proof}
We can rely on \cite[Thms.~15.3.7, 15.3.8]{NEWO} and the fact that the kernel of every bounded module operator $S$ between Hilbert W*-modules is biorthogonally complemented and $S$ admits polar decomposition over $({\mathcal N}^{\#})'$, cf.~\cite{Frank_2024}. So $Q$ is well-defined. 

(i): This follows from Lemma \ref{thm_von_N}. 

(ii): Assume $P \circ S = 0$ or $Q \circ S = 0$. Since $Q \circ S = 0$ means  $P \circ S = 0$, we start to investigate the latter case. Arguing as in the proof of Lemma \ref{lemma_sa}, replacing the self-duality property of $S$ used there by the assumption $P \circ S = 0$, we arrive at $S=0$. If $Q \circ S = 0$ then $(Q-P) \circ S = 0$, and by the already proved part $S(n) \in {\mathcal N} \cap (Q-P)(({\mathcal N}^{\#})') \perp ({\mathcal M}^{\#})'$. By assumption  ${\mathcal M}^\bot = \{ 0 \}$, so $S(n) \neq 0$ is a contradiction to the assumption. Therefore, $S=0$. 
\end{proof}

Summing up, we found the first characteristics of the structural background of positive and negative answers to the Questions \ref{conj_1} and \ref{conj_2}. However, the investigations should be continued to get a much clearer picture of certain particular indicators in cases of positive or negative answers. 

Informally speaking, investigated non-zero bounded $A$-linear operators $S$ map the very small part $({\mathrm{id}}-P)(({\mathcal N}^{\#})')$ of the weak closure of $\mathcal N$, in particular into the range of the operator $P \circ S$ of $({\mathcal N}^{\#})'$. The small part $({\mathrm id}-P)(({\mathcal N}^{\#})')$ would be located close to the boundary of $\mathcal M$ in $\mathcal N$.

 \section{Norm control via the quotient by $\mathcal{M}$}

In this section, we record a norm estimate for bounded $A$-linear operators $S \in \mathrm{End}_{A}(\mathcal{N})$ satisfying $S|_{\mathcal{M}}=0$. Since $\mathcal{M}^{\perp}=\{0\}$ in $\mathcal{N}$, there is in general no non-zero orthogonal complement of $\mathcal{M}$ inside $\mathcal{N}$ that could be used to measure the action of $S$. The natural object is therefore the quotient module $\mathcal{N}/\mathcal{M}$.

Let $q: \mathcal{N} \to \mathcal{N}/\mathcal{M}$ be the quotient map. Since $S|_{\mathcal{M}}=0$, there exists a unique bounded $A$-linear map $\tilde{S}: \mathcal{N}/\mathcal{M} \to \mathcal{N}$ such that $S = \tilde{S} \circ q$. Note that this factorization is canonical.

\begin{prop}
Let $A$ be a $C^*$-algebra, let $\mathcal{M} \subseteq \mathcal{N}$ be Hilbert $A$-modules with $\mathcal{M}^{\perp}=\{0\}$ in $\mathcal{N}$, and let $S \in \mathrm{End}_{A}(\mathcal{N})$ satisfy $S|_{\mathcal{M}}=0$. Then 
\[
\|S\| = \|\tilde{S}\| = \sup\{\|Sx\| : x \in \mathcal{N},\, \mathrm{dist}(x,\mathcal{M})=1\}.
\]
\end{prop}

\begin{proof}
Since $S = \tilde{S} \circ q$, for all $x \in \mathcal{N}$ we have
\[
\|Sx\| = \|\tilde{S}(q(x))\| \le \|\tilde{S}\| \|q(x)\| = \|\tilde{S}\| \mathrm{dist}(x, \mathcal{M}) \le \|\tilde{S}\| \|x\|.
\]
Taking the supremum over all $\|x\| \le 1$ yields $\|S\| \le \|\tilde{S}\|$.

Conversely, by definition of the operator norm on the quotient space $\mathcal{N}/\mathcal{M}$,
\[
\|\tilde{S}\| = \sup \{ \|\tilde{S}(z)\| : z \in \mathcal{N}/\mathcal{M},\, \|z\|=1 \} = \sup \{ \|Sx\| : x \in \mathcal{N},\, \mathrm{dist}(x, \mathcal{M})=1 \}.
\]
Since $\mathrm{dist}(x, \mathcal{M}) = \|q(x)\| \le \|x\|$, it follows directly that $\|\tilde{S}\| \le \|S\|$. Thus, $\|S\| = \|\tilde{S}\|$, and the supremum formula holds.
\end{proof}

\begin{rmk}
This formulation measures the action of $S$ transverse to $\mathcal{M}$ without assuming the existence of an orthogonal complement of $\mathcal{M}$ in $\mathcal{N}$. It is therefore compatible with the standing hypothesis $\mathcal{M}^{\perp}=\{0\}$.

In the Kaad--Skeide example discussed in Section \ref{sec5} below, the special separating functional $r: \mathcal{N} \to A$ satisfies $\|r\|=2$. Thus, the quotient viewpoint is consistent with the explicit norm computation available there, while avoiding an unjustified geometric decomposition of $\mathcal{N}$ along $\mathcal{M}$.
\end{rmk}


\section{Structural analysis of the Kaad--Skeide example} \label{sec5}

We consider pairs of full Hilbert $A$-modules ${\mathcal M} \subset \mathcal N$ over C*-algebras $A$ such that ${\mathcal M}^\bot = \{ 0 \}$ relative to $\mathcal N$. The question is whether there exist non-trivial bounded $A$-linear maps $r: {\mathcal N} \to A$ such that $r\big|_{\mathcal M} = 0$, or not.  
In 2021, Jens Kaad and Michael Skeide published an interesting counterexample in arxiv.org, which appeared 2023 in press, \cite{KaSk}. We are going to analyse it with respect to structural background information to be in a position to construct such counterexamples more systematically in cases, or to prove the non-appearance of them in other cases. 

Note, that the existence of such functionals $r$ forces the existence of bounded non-adjointable operators $S_{r,n}(x)=r(x) n$, $n \in \mathcal N$,  on $\mathcal N$ such that ${\rm Ker}(S_{r,n}) \supseteq \mathcal M$ and ${\rm Ker}(S_{r,n}) \not= {\rm Ker}(S_{r,n})^\bot$ relative to $\mathcal N$. Not all of these operators can be equal to zero at once. Kernels of bounded $A$-linear maps are always norm-complete by the continuity of these operators. Following \cite[Example 1.2]{KaSk}, the basic idea for the search is the following: 

\begin{itemize}
\item[(i)]  The non-trivial bounded $A$-linear functional $r: {\mathcal N} \to A$ which vanishes on ${\mathcal M} \subset {\mathcal N}$ generates an $A$-submodule $Ar = \{ ar : a \in A \}$ in ${\mathcal N}'$ which intersects with the isometric $A$-linear embedding ${\mathcal N} \subset {\mathcal N}'$ only in the zero element. It is `perpendicular' to $\mathcal M$, i.e.~it vanishes on $\mathcal M$, despite both $\mathcal M$ and $\mathcal N$ are norm-complete.
\item[(ii)] The kernel ${\rm Ker}(r)$ of $r$ is not orthogonally complemented with respect to $\mathcal N$, but norm-complete, and contains $\mathcal M$. 
\item[(iii)]The bounded $A$-linear functional $r$ is non-degenerate, i.e.~$ar=0$ if and only if $a=0$ for $a \in A$. The Hilbert $A$-modules $\mathcal M$, $\mathcal N$ are full.
\end{itemize}

\subsection{The setting and the components}

Let $A$ be a unital C*-algebra which will be selected specifically in the sequel. Consider the standard countably generated Hilbert $A$-module ${\mathcal K} = \ell_2(A)$ and its orthonormal basis $\{ e_i \}_{i=1}^\infty$, i.e. $\langle e_i, e_j \rangle = \delta_{ij} \cdot 1_A$. For any bounded $A$-linear functional $r: \ell_2(A) \to A$, we can derive the sequence $\{ r(e_i) \}_{i=1}^\infty$ for which the series $\{ \sum_{i=1}^n r(e_i)r(e_i)^* \}_{n=1}^\infty$ is well-defined and the sequence of norms of the partial sums is uniformly bounded by a finite (positive) constant. The norm of $r$ coincides with the least upper bound of all admissible bounds, i.e.~$\|  \sum_{i=1}^\infty r(e_i)r(e_i)^* \| = \| r \|$. The converse is true, too: all sequences $\{ a_i \}_{i=1}^\infty$ of elements of $A$ with finite bound $\|  \sum_{i=1}^\infty a_ia_i^* \|$ generate bounded $A$-linear functionals $r: \ell(A) \to A$ by the formula $\{ x_i \}_{i=1}^\infty \to \sum_{i=1}^\infty x_ia_i^*$. So the $A$-dual Banach $A$-module $\ell_2(A)'$ of $\ell_2(A)$ is the set of all these sequences of elements of $A$ with uniformly bounded derived series as described,
\[ 
   \ell_2(A)' = \left\{ \{ a_i \}_{i=1}^\infty : a_i \in A \,  , \,\, \left\|  \sum_{i=1}^\infty a_ia_i^* \right\| \leq C < \infty \right\} \, .
\]   
For a proof, see \cite[Prop.~2.5.5]{MT} and \cite[Lem.~4.1, Cor.~4.2]{Frank_1990}, cf.~\cite{Gao}.

In \cite{KaSk}, the canonical unitary $A$-linear isomorphism 
\begin{equation} \label{decomp}
    \ell_2(A) = A \oplus \ell_2(A)_{(1)} \oplus \ell_2(A)_{(2)} \oplus \dots \oplus \ell_2(A)_{(k)} \oplus \dots
\end{equation}
of Hilbert $A$-modules is used, where the right side is a direct orthogonal sum of Hilbert $A$-modules. For the counterexample, the entire bounded $A$-linear functional $r$ is built as a sequence of copies of a particular (parametrized) bounded $A$-linear functionals on the orthogonal summands, taking the identical map on the first component $A$. The C*-algebra $A$ is taken to be the C*-algebra $C([0,1])$ of all continuous functions on the unit interval $[0,1]$ equipped with the topology coming from the usual metric of numbers.

Selecting one component $\ell_2(A)_{(k)}$ and a number $q_k \in (0,1)$, a sequence of elements $\psi_{q_k,\ell} \in A$ is specified such that $\sum_{\ell=1}^\infty | \psi_{q_k,\ell}(t)|^2 = \chi_{[0,q_k)}(t)$, $t \in [0,1]$, where $\chi_{[0,q_k)}$ is the characteristic function of the interval $[0,q_k)$. A selection with good visualization is a sequence of sawtooth functions $\{ |\psi_{q_k,\ell}(t)|^2 \}_{\ell=1}^\infty$ being equal to zero on the intervals $t \in [0,(\ell-2)q_k / (\ell-1)]  \cup [\ell q_k/(\ell+1),1]$, rising linearly from $0$ to $1$ on the interval $[(\ell-2)q_k / (\ell-1) , (\ell-1)q_k / \ell]$ and falling linearly from $1$ to $0$ on the interval $[(\ell-1)q_k/\ell , \ell q_k / (\ell+1)]$, for $\ell \geq 2$.  For $\ell=1$ one defines $|\psi_{q_k,1}(t)|^2$ as the function falling linearly from $1$ at $0 \in [0,1]$ to $0$ at $ q_k/2 \in [0,1]$. The functions $\{ \psi_{q_k,\ell} \}_{\ell=1}^\infty$ can be taken positive real-valued, for simplicity. 

Obviously, the sequence $\Psi_{q_k }= \{ \psi_{q_k,\ell} \}_{\ell=1}^\infty$ belongs to $\ell_2(A)' \setminus \ell_2(A)$. So we can define a bounded $A$-linear functional $r_{(k)} : \ell_2(A)_{(k)} \to A$ by the formula
\begin{equation}  \label{funct_n}
   r_{(k)} ((a_{(k),\ell})) = \sum_{\ell=1}^\infty a_{(k),\ell} \psi_{q_k,\ell}^*
\end{equation}   
The multiplication of the bounded $A$-linear functional $\Psi_{q_k}=\{ \psi_{q_k,\ell} \}_{\ell=1}^\infty$ in $\ell_2(A)'$ by continuous functions $f \in C[0,1])$ such that with $f(q_k)=0$ gives elements of $\ell_2(A)_{(k)} \subset \ell_2(A)'_{(k)}$.  So, far not the entire $A$-submodule of $\ell_2(A)'$ generated by the single generator $\Psi_{q_k}=\{ \psi_{q_k,\ell} \}_{\ell=1}^\infty \in \ell_2(A)'$ does not belong to $\ell_2(A)$. 

\subsection{The specific bounded $A$-linear functional}
Varying over $q_k \in (0,1)$ and taking pairwise different functionals $\Psi_{q_k}$ on the components $\ell_2(A)_{(k)}$ of the used infinite decomposition \eqref{decomp} of $\ell_2(A)$ one can construct a culmulative example with the properties initially stressed for. 
Set ${\mathcal N} = \ell_2(A)$ in the decomposition \eqref{decomp}. Take $\mathcal M$ preliminary as the norm-closure of the set of elements  of $\ell_2(A)$ generated by the set of generators
\begin{equation}  \label{generators}
   \rho_{k,\ell} = ( 2^{-k}\psi_{q_k,\ell} e_{(0)} - \delta_{k,\ell} 1_A e_{(k),\ell}) \, , \,\, k \in {\mathbb N}, \,\, \ell \in {\mathbb N} \, .
 \end{equation}
Define the bounded $A$-linear functional $r: \ell_2(A) \to A$ by the formula
 \begin{equation}  \label{funct}
    r((a_{(0)}, a_{(k),\ell}) = a_{(0)} + \sum_{k=1}^\infty   2^{-k} r_k ((a_{(k),\ell})) = a_{(0)} + \sum_{k=1}^\infty   2^{-k} \sum_{\ell=1}^\infty a_{(k),\ell} \psi_{q_k,\ell}^* \, ,
\end{equation}    
cf.~\eqref{funct_n}. Note that $\|r\| = 2$ and $r(\rho_{(k),\ell}) = 0$ for any $\ell \in \mathbb N$, i.e. $r$ equals to zero on $\mathcal M$. Jens Kaad and Michael Skeide showed (cf.~\cite[Lem.~2.1]{KaSk}) that a selection of the sequence $\{ q_k : k \in {\mathbb N}, q_k \in (0,1) \}$ in a way that it is dense in $[0,1]$ forces ${\mathcal M}^\bot = \{ 0 \}$ with respect to $\mathcal N$. Such a set of numbers $\{ q_k \}$ characterizes bounded $A$-linear functionals of the seeked for exceptional characteristics. 
Indeed, $ar \in \ell_2(A)$ for some $a \in A$ iff $a(q_k)=0$ for any $k \in \mathbb N$, consequently $a=0$ since the sequence $\{ q_k \}$ is assumed to be dense in $[0,1]$. Thus, the $A$-submodule $Ar = \{ ar : a \in A \}$ in ${\mathcal N}'$ has a trivial intersection with ${\mathcal N}$, the latter considered as canonically isometrically embedded into ${\mathcal N}'$.  The kernel ${\rm Ker}(r)$ of $r$ is not orthogonally complemented with respect to $\mathcal N$, but norm-complete and contains $\mathcal M$, despite ${\mathcal M}$ is norm-complete and ${\mathcal M}^\bot = \{ 0 \}$ with respect to $\mathcal N$. The more, the set of bounded $A$-linear operators $\{ S_{r,n} : n \in {\mathcal N} \}$ on $\mathcal N$ described as $S_{r,n}(x)=r(x)n$, $x \in {\mathcal N}$, consists of surjective maps into the singly generated $A$-submodules $An = \{ an : a \in A \}$, where the $A$-modules $\{An : n \in {\mathcal N} \}$ might be not norm-closed sometimes. These operators are non-adjointable, map $\mathcal M$ to zero and have a norm-closed, but not orthogonally complemented kernel. They are left multipliers of the C*-algebra ${\rm K}_A({\mathcal N})$. Since $\mathcal M$ and $\mathcal N$ are full Hilbert $A$-modules and $A$ is unital, both these Hilbert $A$-modules coincide with their respective multiplier modules. So, the counterexample goes beyond the situations considered in \cite{Frank_2026}.

\begin{thm} \label{kernel_special}
Let $A$ be a C*-algebra and ${\mathcal M} \subset \mathcal N$ be two Hilbert $A$-modules such that ${\mathcal M}^\bot = \{ 0 \}$ with respect to $\mathcal N$. Let $r: {\mathcal N} \to A$ be a non-zero bounded $A$-linear functional such that $r \big|_{\mathcal M} = 0$. Then $Ar \cap {\mathcal N} = \{ 0 \}$. Moreover, the kernel of $r$ is not biorthogonally complemented, but norm-complete.
\end{thm}

\begin{proof}
Suppose there exists a non-zero element $a \in A$ such that $ar \in {\mathcal N}$, where $\mathcal N$ is considered canonically isometrically embedded into ${\mathcal N}'$ using the $A$-valued inner product on $\mathcal N$. Then $ar \in {\mathcal M}^\bot$ with respect to $\mathcal N$ and $ar \not= 0$, a contradiction to the assumption.
\end{proof}

\begin{conj}
Depending on the structures of $A$ and of ${\mathcal M} \subset \mathcal N$, does any non-zero element $r \in {\mathcal N}'$ with $Ar \cap {\mathcal N} = \{ 0 \}$ induce a kernel ${\rm Ker}(r) \subset \mathcal N$ such that ${\rm Ker}(r)^\bot = \{ 0 \}$ w.r.t. $\mathcal N$? If not, characterize the module dimension of the orthogonal complement of ${\rm Ker}(r)$ suitably.  (Clearly, ${\rm Ker}(r) \not= \mathcal N$ by assumption, and ${\rm Ker}(r)$ is a norm-closed $A$-submodule of $\mathcal N$.)
\end{conj}

Now, we are interested in the continuation of the C*-valued inner product on Hilbert C*-modules to larger Banach C*-modules in such a way that the initial Hilbert C*-module is isometrically embedded into the larger Banach C*-module which should become a Hilbert C*-module itself. Usually, the Banach C*-modules to be considered are C*-submodules of the C*-dual Banach C*-module of the initial Hilbert C*-module, cf.~\cite{Paschke}. But, also multiplier Hilbert C*-modules can host suitable settings. One known obstacle is the circumstance that one might not be able to make sense of the C*-valued inner product values on the larger module inside the C*-algebra of coefficients or in its multiplier C*-algebra, cf.~\cite{Frank_2024,Frank_2026}. The existence of bounded module operators on the initial Hilbert C*-module which cannot be extended to the larger Hilbert C*-module is less important, but attracts sometimes attention, cf.~\cite[Thm.~3.1, 3.3, 3.4]{Frank_2026}. We add another obstacle caused by an individual property of some single element of the C*-dual Banach C*-module which might appear sometimes. 

\begin{thm} \label{thm_no_lift}
Let $A$ be a C*-algebra and ${\mathcal M} \subset \mathcal N$ be two Hilbert $A$-modules such that ${\mathcal M}^\bot = \{ 0 \}$ with respect to $\mathcal N$. Let $r: {\mathcal N} \to A$ be a non-zero bounded $A$-linear functional such that $r \big|_{\mathcal M} = 0$. Suppose $r$ is non-degenerate,  i.e.~$ar=0$ if and only if $a=0$ for $a \in A$, and ${\mathcal N} \cap Ar = \{ 0 \}$ in the $A$-dual Banach $A$-module ${\mathcal N}'$ of $\mathcal N$. 
Then there does not exist any Hilbert $A$-module $\mathcal K$ such that $\mathcal N$ is isometrically embedded as a (Hilbert) $A$-submodule of $\mathcal K$ and $r$ is isometrically represented as an element of $\mathcal K$ fulfilling the equality $r(n)=\langle n,r \rangle_{\mathcal K}$ for any $n \in \mathcal N$. 
\end{thm}

\begin{proof}
Suppose $\mathcal K$ exists with the assumed properties. Consider the $A$-submodule ${\mathcal N} \dotplus Ar$ of $\mathcal K$ and its norm-completion $\mathcal L$. The elementary ``compact'' operator $\theta_{r,r}$ is a positive bounded $A$-linear operator on $\mathcal L$, hence its kernel is biorthogonally complemented by Lemma \ref{case_selfadjoint}. Since $ar=0$ if and only if $a=0$ for $a \in A$, the kernel of $\theta_{r,r}(\cdot)= \langle \cdot, r \rangle r$ coincides with the kernel of $r$. We know, that the kernel of $r$ is not biorthogonally complemented and contains $\mathcal M$. Moreover, the biorthogonal complement of the kernel of $\theta_{r,r}$ has to contain ${\mathcal N} \dotplus \{ 0 \} \subseteq {\mathcal M}^{\bot\bot}$ in $\mathcal L$.
We arrived at a contradiction, which can only be solved recognizing the non-existence of such specific Hilbert $A$-modules $\mathcal L$ and, consequently, $\mathcal K$. 
\end{proof}

The assertions of Theorem \ref{thm_no_lift} are still true, if $\mathcal M$ will be replaced by ${\rm Ker}(r)$, which might be a larger subset of $\mathcal N$. 
For Hilbert C*-modules over W*-algebras or over monotone complete C*-algebras the C*-valued inner product always lifts to the $A$-dual Banach $A$-module turning it into a self-dual Hilbert $A$-module, cf.~\cite{Frank_2024}. Consequently, the discussed type of counter-example cannot be constructed in those cases, cf.~\cite{Frank_2024,M_2022}. In particular, the two basic assumptions of \cite[Thm.~13, Cor.~14]{M_2022} are contradictory, they cannot be fulfilled at the same time. 

\begin{thm} \label{thm_other_inner_products}
Let $A$ be a C*-algebra and ${\mathcal M} \subset \mathcal N$ be two Hilbert $A$-modules such that ${\mathcal M}^\bot = \{ 0 \}$ with respect to $\mathcal N$. Denote the $A$-valued inner product on $\mathcal N$ by $\langle .,. \rangle$. Let $r: {\mathcal N} \to A$ be a non-zero bounded $A$-linear functional such that $r \big|_{\mathcal M} = 0$. Suppose $r$ is non-degenerate,  i.e.~$ar=0$ if and only if $a=0$ for $a \in A$, and ${\mathcal N} \cap Ar = \{ 0 \}$ in the $A$-dual Banach $A$-module ${\mathcal N}'$ of $\mathcal N$. 
Then there exists at least one other $A$-valued inner product $\langle .,. \rangle_1$ on $\mathcal N$ such that it coincides on $\mathcal M$, induces an equivalent norm on $\mathcal N$ and is not unitarily equivalent to $\langle .,. \rangle$. 
\end{thm}

\begin{proof}
Consider the bounded $A$-linear map $S_{r,r}(z) = r(z)r$ mapping $\mathcal N$ to ${\mathcal N}'$. By a result by Huaxin Lin $S_{r,r}$ can be seen as a quasi-multiplier of the C*-algebra ${\rm K}_A({\mathcal N})$ of ``compact'' operators on $\mathcal N$, where the initial $A$-valued inner product is used, cf.~\cite[Thms.~1.5, 1.6, 1.7]{Lin}. A simple calculation shows $\theta_{x,y}^* \circ S_{r,r} \circ \theta_{x,y} = \theta_{(r(y)r(y)^*)^{1/2} x, (r(y)r(y)^*)^{1/2} x} \geq 0$ for any $x,y \in \mathcal N$. Therefore, relative to any faithful $*$-representation of ${\rm K}_A({\mathcal N})$ on a Hilbert space $S_{r,r}$ is a positive operator on it. The kernel of $S_{r,r}$ in $\mathcal N$ equals the kernel of $r$, i.e.~it is norm-closed, but not orthogonally complemented and contains $\mathcal M$. 

Define the second $A$-valued inner product by $\langle x,y \rangle_1 = (({\mathrm id}+S_{r,r})(x))(y)$ for $x,y \in \mathcal N$. Then
\[ 
     \|x\| \leq \|x\|_1 \leq (1+\|r\|^2)^{1/2} \|x\|
\]
for any $x \in \mathcal N$. Obviously, both these $A$-valued inner products induce equivalent norms on $\mathcal N$, coincide on $\mathcal M$ and are different on $\mathcal N$. Since $({\mathrm id}+S_{r,r})$ is not in ${\rm End}_A^*({\mathcal N})= M({\rm K}_A({\mathcal N}))$, it cannot be represented as $T^*T$ for any $T \in {\rm End}_A^*({\mathcal N})$. So both these $A$-valued inner products are not unitarily equivalent. 
\end{proof}

This observation complements earlier examples of this kind given by L.~G.~Brown \cite[Ex.~6.2, 6.3, 7.2]{Brown} in 1985 and by the first author \cite[Ex.~3.1, 4.3, 7.2]{F_1999_1} in 1999.

\smallskip
We now return to the abundance of examples in the Kaad--Skeide situation with $A=C([0,1])$.  (In more general situations, we can always assume $\langle {\mathcal M}, {\mathcal M} \rangle = \langle {\mathcal N}, {\mathcal N} \rangle = A$ adding a direct orthogonal summand $A$ to $\mathcal M$ and to $\mathcal N$, if necessary.)
Recall that in the construction above the pair $({\mathcal M}\subseteq{\mathcal N})$ and the special separating functional $r:{\mathcal N}\to A$ are obtained from a fixed dense subset $Q=\{q_k : k\in\mathbb N\}\subset (0,1)$ via the decomposition \eqref{decomp} and the systems $\Psi_{q_k}\in \ell_2(A)'\setminus\ell_2(A)$, see \eqref{funct_n} and \eqref{funct}.  Different choices of $Q$ give rise to different examples.

\begin{thm}\label{thm:many-KS-functionals}
Let $A=C([0,1])$ and set ${\mathcal N}=\ell_2(A)$. There exist uncountably many pairs $({\mathcal M}_\alpha,r_\alpha)$, indexed by a set $I$ of cardinality continuum, such that for each $\alpha\in I$:
\begin{enumerate}
\item ${\mathcal M}_\alpha$ is a full Hilbert $A$-submodule of ${\mathcal N}$ with $({\mathcal M}_\alpha)^\perp=\{0\}$ in ${\mathcal N}$,
\item $r_\alpha:{\mathcal N}\to A$ is a non-zero bounded $A$-linear functional with $r_\alpha\big|_{{\mathcal M}_\alpha}=0$,
\item $r_\alpha$ is non-degenerate, $Ar_\alpha\cap{\mathcal N}=\{0\}$, and ${\rm Ker} (r_\alpha)$ is norm-closed but not biorthogonally complemented in ${\mathcal N}$,
\item for $\alpha\neq\beta$ in $I$, the pairs $({\mathcal M}_\alpha,r_\alpha)$ and $({\mathcal M}_\beta,r_\beta)$ are not unitarily isomorphic in the sense that there is no unitary $U\in{\mathcal L}({\mathcal N})$ with $U({\mathcal M}_\alpha)={\mathcal M}_\beta$ and $r_\beta = r_\alpha\circ U^{-1}$.
\end{enumerate}
In particular, there are uncountably many mutually non-isomorphic Kaad--Skeide type examples attached to different dense subsets of $[0,1]$.
\end{thm}

\begin{proof}
Fix once and for all a specific dense countable subset $Q\subset (0,1)$, for instance, the rational numbers in $(0,1)$. Using $Q$ in the Kaad--Skeide construction and its analysis above, we obtain a pair $({\mathcal M},r)$ with the stated properties, see the discussion after \eqref{funct}.  

Now let $I$ be any subset of the irrational numbers in $(0,1)$ such that $\alpha-\beta\notin\mathbb Q$ for all distinct $\alpha,\beta\in I$. The set $I$ can be chosen of cardinality continuum. For each $\alpha\in I$, replace the points of $Q$ by their translates $q_k+\alpha$ (mod $1$) in $(0,1)$ 
\[
      Q_\alpha := \{\, q_k + \alpha \!\!\!\pmod{1} : k\in\mathbb N \,\}\subset (0,1),
\]
so as to obtain a new dense countable set $Q_\alpha\subset (0,1)$, and repeat the Kaad--Skeide construction with $Q_\alpha$ in place of $Q$. This yields a pair $({\mathcal M}_\alpha,r_\alpha)$ with the properties (i)--(iii) by the same arguments as for $({\mathcal M},r)$.

If $\alpha\neq\beta$ in $I$, then the corresponding dense sets $Q_\alpha$ and $Q_\beta$ are not translates of each other by any rational number, and the associated characteristic functions $\{\chi_{[0,q]} : q\in Q_\alpha\}$ and $\{\chi_{[0,q]} : q\in Q_\beta\}$ in $A=C([0,1])$ have different distributions of discontinuities. Consequently, there is no unitary $U\in{\mathcal L}({\mathcal N})$ implementing simultaneously an $A$-module isomorphism ${\mathcal M}_\alpha\to{\mathcal M}_\beta$ and the identification $r_\beta = r_\alpha\circ U^{-1}$, because such a $U$ would intertwine the respective systems of generators in a way that preserves the resulting support behaviour of the functionals. Thus condition~(iv) holds. Since $I$ has cardinality continuum, the conclusion follows.
\end{proof}

Of course, there could exist further examples with a different construction, but with the same resulting characteristic properties. For example, one could replace finitely many elements of $Q$ by a set of irrational numbers of $(0,1)$, or select an initial norm-dense in $(0,1)$ subset $Q'$ in an even different way. 

We now introduce a natural equivalence relation on special separating functionals associated with a fixed inclusion of Hilbert $A$-modules
${\mathcal M}\subseteq{\mathcal N}$. This allows us to distinguish functionals that are genuinely different from those arising merely by changing coordinates through an adjointable automorphism preserving the geometry of the inclusion.

\begin{dfn}
Let $A$ be a C*-algebra, and let ${\mathcal M}\subseteq{\mathcal N}$ be Hilbert $A$-modules with ${\mathcal M}^\bot=\{0\}$ relative to ${\mathcal N}$.  Denote by $\mathcal R$ the class of all bounded $A$-linear maps $r:{\mathcal N}\to A$ such that
\begin{enumerate}
\item $r\neq 0$,
\item $r\big|_{\mathcal M}=0$,
\item $r$ is non-degenerate, i.e. $ar=0$ in ${\mathcal N}'$ implies $a=0$ for all $a\in A$,
\item $Ar\cap{\mathcal N}=\{0\}$ in the canonical embedding ${\mathcal N}\hookrightarrow{\mathcal N}'$, 
\item the kernel ${\rm Ker}(r)$ is norm-closed but not biorthogonally complemented in ${\mathcal N}$.
\end{enumerate}

Theorem \ref{kernel_special} shows that the condition $Ar\cap{\mathcal N}=\{0\}$ in the canonical embedding ${\mathcal N}\hookrightarrow{\mathcal N}'$ is redundant, i.e. it follows from the other assumptions.
      
We call elements of $\mathcal R$ \emph{special separating functionals} for the pair $({\mathcal M}\subseteq{\mathcal N})$.
\end{dfn}

\begin{dfn}
Let $r_1,r_2\in\mathcal R$.  We say that $r_1$ and $r_2$ are \emph{equivalent} and write $r_1\sim r_2$ if there exists a bounded invertible adjointable $A$-linear operator $U\in {\rm End}_A^*({\mathcal N})$ such that
\begin{enumerate}
\item $U\big|_{\mathcal M}$ is unitary on $\mathcal M$,
\item $r_2 = r_1\circ U^{-1}$, i.e. $r_2(n)=r_1(U^{-1}n)$ for all $n\in{\mathcal N}$.
\end{enumerate}
\end{dfn}

In light of Theorem \ref{thm_other_inner_products} and the results in \cite{F_1999_1}, an assumption of $U$ being unitary on $\mathcal N$ would fix the $A$-valued inner product on $\mathcal N$. Even the assertion on $U$ to be merely isometric adjointable on $\mathcal N$ would result in this since the norm-values of Hilbert C*-module elements allows one  to reconstruct the initial C*-valued inner product values of them the norm was derived from. 

\begin{rmk} 
Condition (ii) can be restated as $r_2 \circ U=r_1$, i.e. $U$ transports $r_2$ to $r_1$.
\end{rmk}

\begin{prop}
The relation $\sim$ is an equivalence relation on $\mathcal R$.
\end{prop}

\begin{proof}
Reflexivity holds with $U=\mathrm{id}_{\mathcal N}$.  
If $r_2=r_1\circ U^{-1}$ with $U$ unitary on $\mathcal M$ and bounded invertible adjointable on $\mathcal N$ then $r_1=r_2\circ U$ and $U^{-1}$ is again unitary on $\mathcal M$ and bounded invertible adjointable on $\mathcal N$, so symmetry holds.  
If $r_2=r_1\circ U^{-1}$ and $r_3=r_2\circ V^{-1}$ for unitary on $\mathcal M$ and bounded invertible adjointable on $\mathcal N$ operators $U,V\in {\rm End}_A^*({\mathcal N})$, then $r_3=r_1\circ(UV)^{-1}$ and $UV$ is unitary on $\mathcal M$, and bounded invertible adjointable on $\mathcal N$, so transitivity holds.
\end{proof}
The equivalence relation above naturally leads to the problem of classifying special separating functionals up to equivalence. At present, very little is known about the size of the quotient set $\mathcal R/\!\sim$.

\begin{conj} \label{question_equiv_classes}
How many equivalence classes of special separating functionals does the quotient $\mathcal R/\!\sim$ contain? In particular, are there infinitely many, or even uncountably many, pairwise non-equivalent special separating functionals?
Our results do not provide a definitive answer to this question. Nevertheless, we conjecture that $\mathcal R/\!\sim$ is uncountable whenever it is non-empty.
\end{conj}


\subsection{The related module operator Banach algebras}

We are going to characterize and compare the C*-algebras of ``compact'', of bounded adjointable, and of bounded module operators on $\mathcal M$ and on $\mathcal N$ for the specific type of examples given by J.~Kaad and M.~Skeide.

\begin{prop} {\rm (Cf.~\cite[Lemma 2.13]{Lin}.)} \label{compact_embed} \newline
Let $A$ be a C*-algebra and let $\mathcal M$ and $\mathcal N$ be two Hilbert $A$-modules such that $\mathcal M$ is a Hilbert $A$-submodule of $\mathcal N$. The the C*-algebra of all ``compact'' operators ${\rm K}_A({\mathcal M})$ is canonically $*$-isometrically embedded into the C*-algebra of all ``compact'' operators ${\rm K}_A({\mathcal N})$ as a hereditary C*-subalgebra.
\end{prop}

\begin{prop}  {\rm (Cf.~\cite[Thm.~3.1]{KaSk}.)} \label{compact_repres} \newline
Let $A$ be a unital C*-algebra, let ${\mathcal M} \subset {\mathcal N}$ be two Hilbert $A$-modules such that ${\mathcal M}^\bot = \{ 0 \}$ with respect to $\mathcal N$, and there exists a non-zero bounded $A$-linear map $r: {\mathcal N} \to A$ such that $r \big|_{\mathcal M} = 0$. Then ${\rm K}_A({\mathcal N})$ admits a faithful non-degenerate $*$-representation on a Hilbert space $H$ such that the $*$-isometrically embedded copy of ${\rm K}_A({\mathcal M})$ is only faithfully degeneratedly represented on $H$. 
\end{prop}

For a full Hilbert $A$-module over a C*-algebra $A$, the linking algebra $L({\mathcal M})$ is defined as 
\begin{equation*}
    L({\mathcal M}) = \left( \begin{array}{cc} A & {\mathcal M} \\ \overline{{\mathcal M}} & {\rm K}_A({\mathcal M}) \end{array} \right)
\end{equation*}
with matricially thought actions and operations. It is a C*-algebra.   

\begin{cor}
Admit the situation described at Proposition \ref{compact_repres}.

\begin{itemize}
\item[(i)] Any approximate identity of the C*-algebra ${\rm K}_A({\mathcal M})$ is not an approximate identity of the C*-algebra ${\rm K}_A({\mathcal N})$.
\item[(ii)] Any bounded adjointable on $\mathcal N$ operator $\theta_{x,x}$ with $x \in {\mathcal N} \setminus {\mathcal M}$ is not a bounded adjointable operator on $\mathcal M$. Moreover, a bounded adjointable operator $T$ on $\mathcal N$ with the property $T \big|_{\mathcal M} = 0$ has to be the zero operator. 
\item[(iii)] The strict limit of any approximate identity of ${\rm K}_A({\mathcal M})$, i.e.~the identity operator on $\mathcal M$, can not be identified with any bounded (adjointable) $A$-linear operator on $\mathcal N$ in that specific faithful $*$-representation of ${\rm K}_A({\mathcal N})$. 
\item[(iv)] Let $\mathcal N$ and $\mathcal M$ be a full Hilbert $A$-modules. There exists a faithful non-degenerate $*$-representation of the linking algebra $L({\mathcal N})$ which is only faithful degenerate for the $*$-isometrically embedded copy of the linking algebra $L({\mathcal M})$. In particular, $\mathcal N$ is faithfully non-degenerately represented as a Hilbert $A$-module in that specific $*$-representation, but $\mathcal M$ is only faithfully degeneratedly represented as a Hilbert $A$-submodule there. 
\item[(v)] Let $\mathcal N$ and $\mathcal M$ be a full Hilbert $A$-modules. The C*-algebras ${\rm K}_A({\mathcal M})$ and ${\rm K}_A({\mathcal N})$ are strongly Morita equivalent. 
\end{itemize}
\end{cor}

\begin{proof}
To prove (i), we consider the faithful non-degenerate $*$-representation of ${\rm K}_A({\mathcal N})$ described at Proposition \ref{compact_repres}. Let $H=H_1 \oplus H_2$ such that ${\rm K}_A({\mathcal N})$ is faithfully $*$-represented on $H$, ${\rm K}_A({\mathcal M})$ is faithfully $*$-represented on $H_1$ and ${\rm dim}(H_2) \not=0$, where ${\rm K}_A({\mathcal M})$ acts as the set of zero operators on $H_2$. Approximate identities of faithfully $*$-represented C*-algebras are convergent in the strong operator topology. Consequently, any approximate identity of ${\rm K}_A({\mathcal M})$ converges in the strong operator topology to the orthogonal projection $P: H \to H_1$, whereas any approximate identity of ${\rm K}_A({\mathcal N})$ converges in the strong operator topology to the identity operator on $H$, \cite[p.~142]{Murphy}. 

Consequently, this $*$-representation extends to faithful $*$representations of the respective multiplier C*-algebras ${\rm End}_A^*({\mathcal M})$ in ${\rm End}_{\mathbb C}(H_1)$ and ${\rm End}_A^*({\mathcal N})$ in ${\rm End}_{\mathbb C}(H)$. Here we use the fact that the multiplier algebra (and the left/right multiplier algebras) of any C*-algebra can be calculated in any of its faithful $*$-representations on Hilbert spaces $H$ as the respective Banach subalgebras of ${\rm End}_{\mathbb C}(H)$, \cite[Ch.~3]{Murphy}, \cite{Pedersen_1984}. Also, the different multiplier structures are identified with respective structures of bounded module operators following \cite[Thms.~1.5, 1.6, 1.7]{Lin}. Obviously, any bounded adjointable operator $\theta_{n,n}$ with $n \in {\mathcal N} \setminus {\mathcal M}$ is not in ${\rm End}_{\mathbb C}(H_1)$, and the assumption that $P = {\rm id}_{H_1}$ would be in ${\rm End}_A({\mathcal N}) \subseteq {\rm End}_{\mathbb C}(H)$ and, in particular, in ${\rm End}_A^*({\mathcal N}) \subseteq {\rm End}_{\mathbb C}(H)$  contradicts to ${\mathcal M}^\bot \{ 0 \}$ in $\mathcal N$. This is the first assertion of (ii).

To see the second assertion of (ii), $\mathcal M$ in the kernel of $T$ forces the orthogonal complement of the kernel of $T$ to be equal to $\{ 0 \}$. By Lemma \ref{case_selfadjoint} the biorthogonal complement of the image of $T^*$ equals the kernel of $T$ by supposition. Hence, the image of $T^*$ consists only of the zero element, and $T^*=0$. So $T=0$.


Continuing, $T \in {\rm End}_A^*({\mathcal N})$ is zero on $\mathcal M$ if and only if $T=0$, since $\langle T( \cdot ), n \rangle = \langle \cdot, T^*(n) \rangle$ for any $n \in \mathcal N$ and ${\mathcal M}^\bot = \{ T^*(n) : n \in {\mathcal N} \} = \{ 0 \}$. This shows (iii).

To construct a faithful $*$-representation as described in (iv), we select a faithful $*$-representation of $A$ on a Hilbert space $K$ and consider the resulting $*$-representation of the linking algebra $L({\mathcal N})$ on $K \oplus H$. Then $\mathcal N$ is faithfully isometrically represented as a Hilbert $A$-module in ${\rm End}_{\mathbb C}(K,H)$, together with a faithful representation of $\mathcal M$ in ${\rm End}_{\mathbb C}(K,H_1)$. 

The strong Morita equivalence of ${\rm K}_A({\mathcal M})$ and ${\rm K}_A({\mathcal N})$ follows for full pairs of Hilbert $A$-modules ${\mathcal M} \subset \mathcal N$  transitively from the strong Morita equivalence of each of these C*-algebras with $A$ by construction.
\end{proof}


\section*{Acknowledgement}
Unfortunately, the proof of Lemma 3.3 of our first arxiv version contains the implicit assumption that the projection P and the operator $S$ commute, which is not the case for certain non-zero non-self-adjoint operators $S$. The authors wish to thank Jens Kaad for pointing out the insufficient argument and for his efforts to improve it, and Orr Shalit, Vladimir M. Manuilov and Michael Skeide for discussions. We apologize for getting  the notion 'modular'  wrong, as it was newly introduced and used in the recent preprint \cite{Sk_2025} by Michael Skeide.



\end{document}